\documentclass[12pt,a4paper]{article}
\usepackage{amsthm}
\usepackage{amsmath}
\usepackage{amssymb}
\usepackage{amsfonts}
\newtheorem{thm}{Theorem}[section]

\newtheorem{rem}[thm]{Remark}
\author{}
\title{}
\begin{document}
\begin{center}
{\bf\Large COMMUTATION RELATIONS AND HYPERCYCLIC OPERATORS}
\vskip 0.1cm
{VITALY E. KIM}
\end{center}
\begin{abstract}
In this paper we establish hypercyclicity of continuous linear operators on $H(\mathbb{C})$ that satisfy certain commutation relations.
\end{abstract}
\section{Introduction and statement of the main results}
Let $X$ be a topological vector space. A continuous linear operator $T:X\rightarrow X$ is said to be {\it hypercyclic} if there exists an element $x\in X$ such that it's orbit $\mathrm{Orb}(T,x)=\{T^nx,\;n=0,1,2,...\}$ is dense in $X$. Let us denote by $H(\mathbb{C})$ the space of all entire functions with the uniform convergence topology. For $\lambda\in\mathbb{C}$ we will denote by $S_\lambda$ the translation operator on $H(\mathbb{C})$: $S_\lambda f(z)\equiv f(z+\lambda)$. The first examples of hypercyclic operators are due to G. D. Birkhoff and G. MacLane. In 1929, Birkhoff \cite{1} showed that the translation operator $S_\lambda$ is hypercyclic on $H(\mathbb{C})$ if $\lambda\neq0$. In 1952, MacLane \cite{10} proved that the differentiation operator is hypercyclic on $H(\mathbb{C})$. A significant generalization of these results was proved by G. Godefroy and J. H. Shapiro \cite{3} in 1991. They showed that every convolution operator on $H(\mathbb{C}^N)$ that is not a scalar multiple of the identity is hypercyclic. Recall that a convolution operator on $H(\mathbb{C}^N)$ is a continuous linear operator that commutes with all translations (or equivalently, commutes with each partial differentiation operator).
\par
Some classes of non-convolution hypercyclic operators on $H(\mathbb{C})$ (or $H(\mathbb{C}^N)$) can be found in \cite{16}, \cite{22}, \cite{17}, \cite{15}, \cite{18}. More information on hypercyclic operators can be found in \cite{19}, \cite{13}, \cite{20}.
\par
In the present paper we prove the following result.
\begin{thm}\label{th1}
Let $T:H(\mathbb{C})\rightarrow H(\mathbb{C})$ be a continuous linear operator that satisfies the following conditions:
\begin{enumerate}
	\item $\mathrm{ker}\,T\neq\{0\}$;
	\item $T$ satisfies the commutation relation:
\begin{equation}\label{f1}
	[T,\frac{\mathrm{d}}{\mathrm{d}z}]=aI,
\end{equation}
where $a\in\mathbb{C}\setminus\{0\}$ is some constant and $I$ is identity operator.
\end{enumerate}
Then $T$ is hypercyclic.
\end{thm}
\begin{rem}
In Theorem \ref{th1} we consider a constant $a\neq0$. Note that if $a=0$ in (\ref{f1}) then $T$ is a convolution operator on $H(\mathbb{C})$ that is hypercyclic by a Theorem of Godefroy and Shapiro \cite{3}.
\end{rem}
We also prove the following more general result.
\begin{thm}\label{th2}
Let $T:H(\mathbb{C})\rightarrow H(\mathbb{C})$ satisfies the conditions of Theorem \ref{th1} and let $L$ be an entire function such that $L\not\equiv\mathrm{const}$ and $L(T)$ is a continuous operator from $H(\mathbb{C})$ to itself. Then $L(T)$ is hypercyclic. 
\end{thm}
The paper is organised as follows. In Section 2 we formulate and prove a theorem that establishes the completeness of translates of certain entire functions. In Section 3 we prove Theorems \ref{th1} and \ref{th2} by application of the results of Section 2 and the Godefroy-Shapiro criterion. In Section 4 we give a description of operators that are hypercyclic by Theorem \ref{th1} or \ref{th2} and provide some examples.
\section{Complete systems of translates of entire functions}
Recall that the system of entire functions $\{f_\lambda(z),\lambda\in\Lambda\subset\mathbb{C}\}$ is called complete in $H(\mathbb{C})$ if $\overline{\mathrm{span}\{f_\lambda,\lambda\in\Lambda\}}=H(\mathbb{C})$. The connection between complete systems of entire functions and hypercyclic operators on $H(\mathbb{C})$ is established by the Godefroy-Shapiro criterion which was exhibited by G. Godefroy and J. H. Shapiro in \cite{3} (see also e.g. \cite[Ch. 1]{13}):
\begin{thm}[Godefroy-Shapiro criterion]\label{th4}
Let $T$ be a continuous linear operator on $H(\mathbb{C})$. Suppose that there is a system of entire functions $\{f_\lambda\}_{\lambda\in\mathbb{C}}$ such that:
\begin{enumerate}
	\item $T[f_\lambda]=\lambda f_\lambda$, $\forall\lambda\in\mathbb{C}$;
	\item both systems $\{f_\lambda,\,\lambda\in\mathbb{C}:|\lambda|<1\}$ and $\{f_\lambda,\,\lambda\in\mathbb{C}:|\lambda|>1\}$ are complete in $H(\mathbb{C})$.
	\end{enumerate}
Then $T$ is hypercyclic
\end{thm}
In this section we prove the following result.
\begin{thm}\label{th23}
Let $T:H(\mathbb{C})\rightarrow H(\mathbb{C})$ satisfies the conditions of Theorem \ref{th1} for some $a\neq0$. Then the system $\{S_\lambda f\}_{\lambda\in\Lambda}$ is complete in $H(\mathbb{C})$ for any function $f\in\mathrm{ker}\,T\setminus\{0\}$ and any set $\Lambda\subset\mathbb{C}$ that has an accumulation point in $\mathbb{C}$.
\end{thm}
\begin{proof}
Let $T$ satisfy the conditions of Theorem \ref{th1}. Then there is an entire function $f$ such that $f\in\mathrm{ker}\,T$  and $f\not\equiv0$. From (\ref{f1}) it follows that
$$[T,\frac{\mathrm{d}^n}{\mathrm{d}z^n}]=an\frac{\mathrm{d}^{n-1}}{\mathrm{d}z^{n-1}},\;n=0,1,...$$
(see e.g. \cite[\S 16]{23}). Hence,
\begin{equation}\label{f3}
	Tf^{(n)}=anf^{(n-1)},\;n=0,1,...
\end{equation}
Recall that every convolution operator $M_L$ on $H(\mathbb{C})$ can be represented as a linear differential operator of, generally speaking, infinite order with constant coefficients: 
\begin{equation}\label{eq1}
M_L[f](z)=\sum_{k=0}^\infty d_kf^{(k)}(z),
\end{equation}
where $L(\lambda)=\sum_{k=0}^\infty d_k\lambda^k$ is an entire function of exponential type. The function $L$ is called a characteristic function of the convolution operator (\ref{eq1}). A convolution operator (\ref{eq1}) can be also represented in the form
\begin{equation}\label{eq101}
M_L[f](\lambda)=(F_L,S_\lambda f),
\end{equation}
where $F_L$ is a continuous linear functional on $H(\mathbb{C})$ such that $(F_L,\exp(\lambda z))=L(\lambda)$ (see e.g. \cite{24}).
\par
Let $\{\lambda_j\}_{j=1}^\infty$ be the set of zeros of $L$. Denote by $m_j$ the multiplicity of zero $\lambda_j$. Then $\mathrm{ker}\,M_L$ contains exponential monomials $z^b\exp(\lambda_j z)$, $b=0,1, ...,m_j-1$, $j\in\mathbb{N}$ and does not contain any other exponential monomials. In particular, $\mathrm{ker}\,M_L$ does not contain any exponential monomial of the kind $z^p\exp(\mu z)$ if $L(\mu)\neq0$. By the result of L. Schwartz \cite{14}, every differentiation-invariant closed linear subspace $V$ of $H(\mathbb{C})$ coincides with the closure of the linear span of exponential monomials that are contained in $V$. Hence,
$$\mathrm{ker}\,M_L=\overline{\mathrm{span}\{z^b\exp(\lambda_j z),\,j\in\mathbb{N},\,b=0,1, ...,m_j-1\}}.$$
Denote by $\mathrm{Conv}[H(\mathbb{C})]$ the class of all convolution operators (\ref{eq101}) on $H(\mathbb{C})$ except the operator of multiplication by zero.
\par
Let $\Lambda$ be an arbitrary subset of $\mathbb{C}$ that has an accumulation point in $\mathbb{C}$. In order to prove that the system $\{S_\lambda f\}_{\lambda\in\Lambda}$ is complete in $H(\mathbb{C})$, it is enough to prove that $f\notin\mathrm{ker}\,M$ for any $M\in\mathrm{Conv}[H(\mathbb{C})]$. Indeed, if the system $\{S_\lambda f\}_{\lambda\in\Lambda}$ is not complete in $H(\mathbb{C})$, then according to the Approximation Principle (see e.g. \cite[Ch. 2]{25}) there is a non-zero continuous linear functional $F$ on $H(\mathbb{C})$ such that $(F,S_\lambda f)=0$, $\forall\,\lambda\in\Lambda$ and hence, $f\in\mathrm{ker}\,M$ for some $M\in\mathrm{Conv}[H(\mathbb{C})]$.
\par
Assume that $f\in\mathrm{ker}\,M_\Phi$ for some $M_\Phi\in\mathrm{Conv}[H(\mathbb{C})]$ with the characteristic function $\Phi(\lambda)=\sum_{k=0}^\infty b_k\lambda^k$, i.e.
$$\sum_{k=0}^\infty b_kf^{(k)}(z)\equiv0.$$
Then, obviously, $T^nM_\Phi f=0$, $\forall n\in\mathbb{N}$. Using (\ref{f3}) we obtain
$$T^nM_\Phi[f](z)\equiv a^n\sum_{k=n}^\infty b_k\frac{k!}{(k-n)!}f^{(k-n)}(z).$$
Hence,
\begin{equation}\label{f4}
\sum_{k=n}^\infty b_k\frac{k!}{(k-n)!}f^{(k-n)}(z)\equiv0,\,n=0,1,\cdots.	
\end{equation}
We see that $f$ satisfies the infinite system of convolution equations (\ref{f4}) with the characteristic functions $\Phi^{(n)}$, $n=0,1,\cdots$. Let us note that $\Phi\not\equiv0$, since by assumption $M_\Phi\in\mathrm{Conv}[H(\mathbb{C})]$.
\par
Denote $W=\bigcap_{n=0}^\infty\mathrm{ker}\,M_{\Phi^{(n)}}$. It is easy to see that $W$ is differentiation-invariant closed linear subspace of $H(\mathbb{C})$. Assume that $W$ contains some exponential monomial of the kind $z^p\exp(\mu z)$. Then $\Phi^{(n)}(\mu)=0$, $n=0,1,\cdots$. But this is impossible, since by assumption $\Phi\not\equiv0$. Hence, $W$ does not contain any functions of the kind $z^p\exp(\mu z)$. Then, by the result of L. Schwartz \cite{14}, $W$ is trivial, i.e. $W=\{0\}$. We have a contradiciton to the fact that $f\in W$ and $f\not\equiv0$. Hence, $f\notin\mathrm{ker}\,M_\Phi$.
\end{proof}
\begin{rem}
In the proof of Theorem \ref{th23} we use the result of L. Schwartz \cite{14} on spectral synthesis in $H(\mathbb{C})$. But we cannot prove the analogue of Theorem \ref{th23} for $H(\mathbb{C}^N)$ $(N>1)$ using the same method because there is a conter-example on spectral synthesis in $H(\mathbb{C}^N)$ by D. Gurevich \cite{26}. In particular Gurevich showed that there is a non-trivial differentiation-invariant subspace $V$ in $H(\mathbb{C}^N)$ that does not contain any exponential monomial. We are going to discuss the analogues of Theorems \ref{th23}, \ref{th1} and \ref{th2} for $H(\mathbb{C}^N)$ in one of the next papers.
\end{rem}
\section{Proofs of Theorems \ref{th1} and \ref{th2}}
It is enough to prove Theorem \ref{th2}, since Theorem \ref{th1} is a special case of Theorem \ref{th2}, when $L(z)\equiv z$. Let $T:H(\mathbb{C})\rightarrow H(\mathbb{C})$ and $L\in H(\mathbb{C})$ satisfy the conditions of Theorem \ref{th2}. Let $f\in\mathrm{ker}\,T\setminus\{0\}$. Using the Taylor expansion we can write:
$$S_\lambda f(z)=\sum_{n=0}^\infty\frac{f^{(n)}(z)}{n!}\lambda^n.$$	
Then from (\ref{f3}) it follows that 
$$TS_\lambda f=\sum_{n=0}^\infty\frac{T[f^{(n)}(z)]}{n!}\lambda^n=\sum_{n=1}^\infty\frac{anf^{(n-1)}(z)}{n!}\lambda^n=a\lambda
\sum_{n=0}^\infty\frac{f^{(n)}(z)}{n!}\lambda^n.$$
Hence, $TS_\lambda f=a\lambda S_\lambda f$ and $L(T)f=L_a(\lambda)S_\lambda f$, $\forall\lambda\in\mathbb{C}$, where $L_a(\lambda)\equiv L(a\lambda)$. Since $L_a\not\equiv\mathrm{const}$, then both sets $\{\lambda\in\mathbb{C}:|L_a(\lambda)|<1\}$ and $\{\lambda\in\mathbb{C}:|L_a(\lambda)|>1\}$ are non-empty open subsets of $\mathbb{C}$. Hence, $L(T)$ is hypercyclic by Theorems \ref{th23} and \ref{th4}.

\section{Operators that satisfy Theorem \ref{th1} or \ref{th2}}
In this section we give the examples of linear continuous operators on $H(\mathbb{C})$ which are hypercyclic by Theorem \ref{th1} or \ref{th2}. First, we shall describe the class of linear continuous operators on $H(\mathbb{C})$ that satisfy the condition 2 of Theorem \ref{th1}.
\begin{thm}\label{th5}
A continuous linear operator $T:H(\mathbb{C})\rightarrow H(\mathbb{C})$ satisfies the commutation relations (\ref{f1}) if and only if $T=M-azI$, where $M$ is a convolution operator (\ref{eq1}) on $H(\mathbb{C})$.
\end{thm}
\begin{proof}
1) Let $T=M-azI$, where $M$ is an arbitrary convolution operator (\ref{eq1}) on $H(\mathbb{C})$. Then we have $[T,\frac{\mathrm{d}}{\mathrm{d}z}]=[M,\frac{\mathrm{d}}{\mathrm{d}z}]-[azI,\frac{\mathrm{d}}{\mathrm{d}z}]=aI$.
\par
2) Let $T$ satisfies the commutation relation (\ref{f1}), i.e. $[T,\frac{\mathrm{d}}{\mathrm{d}z}]=aI$. Let us define the operator $T_0=-azI$. Then $[T_0,\frac{\mathrm{d}}{\mathrm{d}z}]=aI$. Denote $M=T-T_0$. Then $[M,\frac{\mathrm{d}}{\mathrm{d}z}]=[T-T_0,\frac{\mathrm{d}}{\mathrm{d}z}]=0$. Hence, $M$ is a convolution operator on $H(\mathbb{C})$, and $T=T_0+M=M-azI$.
\end{proof}
Let $T=M-azI$, where $M$ is a convolution operator. Then $T$ is hypercyclic by Theorems \ref{th1} and \ref{th5} if there exists an entire function $f$ such that $f\not\equiv0$ and $f\in\mathrm{ker}\,T$. Here we provide some examples of such functions and corresponding hypercyclic operators.
\par\noindent
{\bf Example}
a) Let $T=\frac{d}{dz}-zI$. If $f(z)=Ce^{z^2/2}$, where $C$ is some constant, then $f\in\mathrm{ker}\,T$. Hence, T is hypercyclic by Theorem \ref{th1}. Consider the operator
\begin{equation}\label{f5}
L(T)=\frac{d^2}{dz^2}+(2z-1)\frac{d}{dz}+(z^2-z-1)I,	
\end{equation}
where $L(\lambda)=\lambda^2+\lambda$. Then the operator (\ref{f5}) is hypercyclic by Theorem \ref{th2}.
\par
b) Let $T=\frac{d^2}{dz^2}-zI$. If $f(z)=C_1\mathrm{Ai}(z)+C_2\mathrm{Bi}(z)$, where $C_1$, $C_2$ are some constants, and $\mathrm{Ai}$, $\mathrm{Bi}$ are Airy functions (see e.g. \cite{11}), then $f\in\mathrm{ker}\,T$. Hence, T is hypercyclic by Theorem \ref{th1}.

Now we will describe the class of hypercyclic differential operators of infinite order with polynomial coefficients. This class contains as a partial case all convolution operators on $H(\mathbb{C})$ that are not scalar multiples of the identity.
\begin{thm}\label{th6}
Let $T=\frac{d}{dz}-azI$, where $a\in\mathbb{C}$ is some constant. Let $L(z)=\sum_{k=0}^\infty d_k\lambda^k$ be an arbitrary entire function of exponential type such that $L\not\equiv\mathrm{const}$. Then the operator
$$L(T)=\sum_{k=0}^\infty d_k T^k$$
is hypercyclic operator on $H(\mathbb{C})$.
\end{thm}
\begin{proof}
Let's show that $L(T)f\in H(\mathbb{C})$, $\forall f\in H(\mathbb{C})$. We can represent $f$ in the form $f(z)=e^{az^2/2}g(z)$. Then $T^kf(z)=e^{az^2/2}g^{(k)}(z)$, $k=0,1,\cdots$. Hence, $L(T)f(z)=e^{az^2/2}M_Lg(z)\in H(\mathbb{C})$, where $M_L$ is a convolution operator on $H(\mathbb{C})$ generated by $L$. Since $T$ satisfies the conditions of Theorem \ref{th1}, then $L(T)$ is hypercyclic by Theorem \ref{th2}.
\end{proof}

\par\noindent
Vitaly E. Kim
\par\noindent
Institute of Mathematics with Computing Centre,
\par\noindent
Russian Academy of Sciences,
\par\noindent
112 Chernyshevsky str.,
\par\noindent
450008 Ufa, Russia
\par\noindent
e-mail: kim@matem.anrb.ru
\end{document}